\documentclass[12pt]{article}
\usepackage{amsmath,amsthm}
\usepackage{amssymb,latexsym}

\newtheorem{thm}{Theorem}[section]
\newtheorem{cor}[thm]{Corollary}
\newtheorem{lem}[thm]{Lemma}

\theoremstyle{definition}
\newtheorem{defin}[thm]{Definition}
\newtheorem{rem}[thm]{Remark}

\numberwithin{equation}{section}

\begin{document}

\title{Spectral Factorization and Lattice Geometry}

\author{Wayne Lawton\\
Department of Mathematics\\
Mahidol University, Bangkok, Thailand \\
and School of Mathematics and Statistics\\
University of Western Australia, Perth, Australia\\
E-mail: wayne.lawton@uwa.edu.au}

\date{}

\maketitle

\renewcommand{\thefootnote}{}

\footnote{2010 \emph{Mathematics Subject Classification}: Primary 11P21; Secondary 42B05.}

\footnote{\emph{Key words and phrases}: Fej\'{e}r-Riesz spectral factorization, holomorphic function, Bohr subset, modular group, Liouville-Roth constant, Diophantine approximation}

\renewcommand{\thefootnote}{\arabic{footnote}}
\setcounter{footnote}{0}

\begin{abstract}
We obtain conditions for a trigonometric polynomial $t$ of one variable to equal or be approximated by $|p|^2$ where $p$ has frequencies in a Bohr set of integers obtained by projecting lattice points in the open planar region bounded by the lines $y = \alpha x \pm \beta$ where $|\beta| \leq \frac{1}{4}$ and $\alpha$ is either rational or irrational with Liouville-Roth constant larger than $2.$ We derive and use a generalization of the Fej\'{e}r-Riesz spectral factorization lemma in one dimension, an approximate spectral factorization in two dimensions, the modular group action on the integer lattice, and Diophantine approximation.
\end{abstract}

\section{Introduction and Preliminary Results}

We introduce notation and summarize standard results that we will use throughout the paper.
$\mathbb{Z}_{+}, \mathbb{Z}, \mathbb{Q}, \mathbb{R}, \mathbb{C}$ denote the nonnegative integer,
integer, rational,
real, and complex numbers, $\mathbb{T} = \mathbb{R}/\mathbb{Z}$ is the real circle
group, $\mathbb{T}_c = \{z\in \mathbb{C}:|z|=1\}$ is the complex circle group,
and the maps $e_j : \mathbb{T} \rightarrow \mathbb{T}_c, j \in \mathbb{Z}$
defined by $e_j(x) = e^{2\pi i j x}, \ x \in \mathbb{T}$ are homomorphisms.
For $z = x+iy, x,y \in \mathbb{R}$ we define $\Re\, z = x$ and $\Im\, z = y.$
If $z \neq 0$ then there exist unique $r > 0$ and $-\pi < \theta \leq \pi$
such that $z = re^{i\theta}$ and we define $\log z = \log r + i \theta$ and
$\sqrt z = {\sqrt r} e^{i\theta/2}.$

The closed unit disk $\mathbb{D} = \{z\in \mathbb{C}:|z| \leq 1\}.$

For $\sigma > 0,$
$
    \mathbb{A}_\sigma = \{ z \in \mathbb{C}: e^{-\sigma} \leq |z| \leq e^{\sigma} \}
$
is a closed annulus.
For a topological space $X,$ $C(X)$ denotes the Banach algebra of bounded
continuous functions $f:X\rightarrow \mathbb{C}$ with norm
$||f||_{\infty} = \sup \, \{\, |f(x)| : x \in X\, \}.$ For $X \subset \mathbb{C},$ $X^{o}$ denotes its interior and
$H(X)$ denotes the Banach subspace of $C(X)$ consisting of functions whose restriction to $X^{o}$ is
holomorphic (complex analytic). For $f \in C(\mathbb{T})$ we define $||f||_1 = \int_{x \in \mathbb{T}} |f(x)|dx,$ the \emph{Fourier transform} $\widehat f \in C(Z)$ by
${\widehat f}\, (j) = \int_{x \in \mathbb{T}} f(x) \, e_{-j}(x) \, dx,\ j \in \mathbb{Z},$
and the \emph{frequency set}
$\hbox{freq}(f) = \{j \in \mathbb{Z}:{\widehat f}(j) \neq 0\}.$
If $f, h \in C(T)$ then $\hbox{freq}(\overline f) = -\hbox{freq}(f)$ and $\hbox{freq}(fh) \subseteq \hbox{freq}(f)+\hbox{freq}(h)$ and hence $\hbox{freq}(|f|^2) \subseteq \hbox{freq}(f) - \hbox{freq}(f).$
The convolution $f*h$ of $f, h \in C(T),$ defined by
$(f*h)(x) = \int_{y\in \mathbb{T}} f(y)h(x-y)dx,$
satisfies
$||f*h||_\infty \leq ||f||_1\, ||h||_\infty$
and
$\widehat {f*h} = {\widehat f}\ {\widehat h}.$
We let $\ell^{1}(\mathbb{Z})$ denote the Banach space of absolutely summable
$f \in C(\mathbb{Z})$ with norm
$||f||_{1} = \sum_{j \in \mathbb{Z}} |f(j)| < \infty$
and we let $A(T)$ denote the \emph{Wiener algebra} of $f \in C(T)$ such that $\widehat f \in \ell^1(\mathbb{Z}).$
We let $C^{\omega}(T)$ denote the algebra of \emph{real analytic} functions on $T$ and we let $\mathfrak{T}_1$ denote
the algebra of \emph{trigonometric polynomials} that consists of $f \in C(T)$ such that $\hbox{freq}(f)$ is finite.
For $t \in \mathfrak{T}_1$ define $n^{-}(t) = \min(\hbox{freq}(t)),$ $n^{+}(t) = \max(\hbox{freq}(t)),$ $n(t) = \max\{n^+(t),n^-(t)\}.$
For any $F \subset \mathbb{Z},$ we define
$
    \mathfrak{T}_1(F) =  \{ f \in \mathfrak{T}_1 : \hbox{freq}(f) \subseteq F\}.
$
Clearly $\mathfrak{T}_1 \subset C^{\omega}(T) \subset A(T) \subset C(T)$ and these algebras are closed under the involution $f \rightarrow \overline f.$
We define the algebras $A^{\pm}(\mathbb{T}) = \{f \in A(\mathbb{T}) : \hbox{freq}(f) \subseteq \pm \mathbb{Z}_+\}$ and observe that they are not closed under involution.
For $N \in \mathbb{Z}_+$ we define
\emph{Dirichlet}, \emph{Hilbert} and \emph{analytic} kernels by
$D_N = \sum_{j=-N}^{N} e_j,$ $H_N = \sum_{j=1}^{N} (e_j - e_{-j}),$ and $A_N^{\pm} = \frac{1}{2}(D_N \pm H_N).$
If $f \in C(T)$ then $D_N*f = \sum_{j=-N}^N {\widehat f}(j)\, e_j$ and in general $||f - D_N*f||_\infty$ does not converge to $0.$
If $f \in A(T)$ then $||f-D_N*f||_{\infty} \rightarrow 0$ as $N$ increases and its rate of convergence increases with smoothness of $f.$ The following approximation error bound holds {\cite[Chapter 10, Exercise 24 ]{Rudin2}}.
\begin{lem}
\label{lem1}
    $f \in C^{\omega}(T)$ if and only if there exists $\sigma > 0$ and $F \in H(\mathbb{A}_\sigma)$
    such that $f = F(e_1).$ Then for every $N \geq 0,$
\begin{equation}
\label{e1}
    ||f - D_N*f||_{\infty} < 2 \, ||F||_{\infty} \, \sigma^{-1} \, e^{-N\sigma}.
\end{equation}
\end{lem}
\begin{proof}
The first assertion follows since real analytic functions are closed under composition. If $F \in H(\mathbb{A}_\sigma)$ then Cauchy's integral formula
$$F(z) = \frac{1}{2\pi i} \left( \int_{|w| = e^{\sigma}} - \int_{|w| = e^{-\sigma}} \right)\,  \frac{F(w)}{w-z}dw, \ \ z \in \mathbb{A}_\sigma^{o},$$
implies that $F$ has the Laurent expansion
$F(z) = \sum_{j \in \mathbb{Z}} c_j \, e^{-|j|\sigma} \, z^j$ where
$c_j = \int_{0}^{1} F(e^{\sigma} \, e_1(x)) \, e_{-j}(x)\, dx, \, j \geq 0;$
$c_j = \int_{0}^{1} F(e^{-\sigma} \, e_1(x)) \, e_{-j}(x)\, dx, \, j < 0.$
Therefore, if $f = F(e_1)$ then
$||f-D_N*f||_{\infty} \leq 2 \, ||F||_{\infty}\, e^{-(N+1)\sigma}/(1-e^{-\sigma})$
and the second assertion follows since $e^\sigma > 1+\sigma.$
\end{proof}
\begin{lem}
\label{lem2}
If $f \in A^{\pm}(\mathbb{T})$
then $\exp(f) \in A^{\pm}(\mathbb{T}).$ Furthermore, for any $n, N \in \mathbb{Z}_+$ with $n \leq N,$
\begin{equation}
\label{e2}
    (1/2+A_{n}^{\pm})*\exp(f) = (1/2+A_{n}^{\pm})*\exp((1/2+A_{N}^{\pm})*f).
\end{equation}
\end{lem}
\begin{proof} The first assertion follows since $A(T)$ is a Banach algebra under the norm $||f||_{A(\mathbb{T})} = ||\widehat f||_1.$
Let $g =  (1/2+A_{N}^{\pm})*f.$ Since $f - g = \sum_{j=N+1}^{\infty} {\widehat f}(j) \, e_{\pm j},$
$\exp(f - g) = 1 + \sum_{j=N+1}^{\infty} c_j \, e_{\pm j}$ where $c_j$ is absolutely summable.
Therefore $\exp(f) = \exp(g) + \sum_{j=N+1}^{\infty} d_j \, e_{\pm j}$ where $d_j$ is absolutely summable
and hence (\ref{e2}) holds.
\end{proof}
\begin{lem}
\label{lem3} The kernels defined above satisfy the following inequalities.
\begin{equation}
\label{e3}
    ||D_N||_{1} \leq 1 + \log (2N+1).
\end{equation}
\begin{equation}
\label{e4}
    ||A_N^{\pm}||_{1} \leq \frac{3}{2} + \log (N) \ \hbox{and} \  ||1/2 + A_N^{\pm}||_{1} \leq 1 + \log (N+1).
\end{equation}
\begin{equation}
\label{e5}
    ||H_N||_{1} \leq 1 + 2\log (N).
\end{equation}
\end{lem}
\begin{proof} We reproduce the derivation of (\ref{e3}) given in {\cite[Section 16.2]{Powell}}. \\
Since $D_N^{\pm}(0) = 2N+1$ and
$|D_N^{\pm}(x)| = \frac{\sin(\pi(2N+1)x}{\sin \pi x}, x \in (0,1/2]$ it follows that $|D_N(x)| \leq \min \{2N+1, \frac{1}{2x} \}.$
Therefore for
$0 < u \leq \frac{1}{2},$
$$
    ||D_N||_{1} \leq 2 \int_{0}^{u} (2N+1) dx
    + 2 \int_{u}^{\frac{1}{2}} \frac{1}{2x} dx = 2(2N+1)u - \log 2 - \log u.
$$
Minimizing the function of $u$ on the right gives (\ref{e3}). The inequalities in (\ref{e4}) follow from a similar argument
and (\ref{e5}) follows since $H_N = A_{N}^{+}-A_{N}^{-}.$
\end{proof}
For $d \geq 1,$
$\mathfrak{T}_d$ is the algebra of trigonometric polynomials on $\mathbb{T}^d,$ and
$\mathfrak{T}_d^{+}  = \{t \in \mathfrak{T}_d : t \geq 0\}.$
For $F \subset \mathbb{Z}^d,$
$\mathfrak{T}_d(F) = \{ t \in \mathfrak{T}_d : \hbox{freq}(t) \subseteq F\},$
$\mathfrak{T}_d^{+}(F) = \mathfrak{T}_{d}^{+} \cap \mathfrak{T}_d(F),$
$\mathfrak{S}_d(F) = \{ |p|^2  :  p \in \mathfrak{T}_d(F) \},$
$\mathfrak{U}_d(F) = \mathfrak{T}_{d} \cap \overline {\mathfrak{S}_{d}(F)}.$
\begin{rem}
Clearly $\mathfrak{S}_d(F) \subseteq  \mathfrak{T}_d^{+}(F-F)$ but equality does not generally hold
since $t(x,y) = 5 + 2 \cos 2\pi x + 2 \cos 2 \pi y$ is positive but irreducible.
The Fej\'{e}r-Riesz spectral factorization lemma (\cite{Riesz-Nagy}, p. 117),
conjectured by Fej\'{e}r \cite{Fejer} and proved by Riesz \cite{Riesz}, shows that if
$F = \{0,...,n\}$ then $\mathfrak{S}_d(F) =  \mathfrak{T}_d^{+}(F-F).$
If $F$ is the set of lattice points in a half-space in $\mathbb{R}^d$ then $F-F =\mathbb{Z}^d$ and
results of Helson and Lowdenslager \cite{Helson} imply that
$\mathfrak{U}_{d}(F)  = \mathfrak{T}_d^{+}.$
Rudin \cite{Rudin1}, Dritschel \cite{Dritschel}, Geronimo and Woerdneman \cite{Geronimo}, and others have extended these results.
\end{rem}
\noindent For $\alpha \in \mathbb{R}$ define $\theta_\alpha \, : \, \mathbb{Z} \rightarrow \mathbb{Z}$ and
$\Theta_\alpha : \mathfrak{T}_1 \rightarrow \mathfrak{T}_2$ by
\begin{equation}
\label{theta}
\noindent \theta_\alpha(j) = \hbox{ smallest integer that minimizes } |\theta_\alpha(j) - j\alpha|, \ j \in \mathbb{Z},
\end{equation}
\begin{equation}
\label{Theta}
\noindent \Theta_\alpha(t)(x,y) = \sum_{j \in \mathbb{Z}} {\widehat t}(j) \, e_j(x) \, e_{\theta_\alpha(j)}(y), \ t \in \mathfrak{T}_1.
\end{equation}
\noindent For $\alpha \in \mathbb{R}$ and $\beta > 0$ define $F_1(\alpha,\beta) \subseteq \mathbb{Z}$ and $F_2(\alpha,\beta) \subseteq \mathbb{Z}^2$ by
\begin{equation}
\label{F1}
    F_1(\alpha,\beta)=\{ j \in \mathbb{Z}  :  |\theta_\alpha(j)-j\alpha| < \beta \},
\end{equation}
\begin{equation}
\label{F2}
    F_2(\alpha,\beta)=\{ (j,k) \in \mathbb{Z}^2 : |k-j\alpha| < \beta  \}.
\end{equation}
\noindent Clearly the map $j \rightarrow (j,\theta_\alpha(j))$ maps $F_1(\alpha,\beta)$ into $F_2(\alpha,\beta)$ and projection
onto the first coordinate is its left inverse.
\begin{rem}
$F_1(\alpha,\beta)$ are \emph{Bohr sets}, named after Harold Bohr whose work on almost periodic functions \cite{Bohr}
they relate to. They arise in combinatorial number theory (\cite{Geroldinger}, p. 132-137) and in symbolic dynamics
since if $\alpha$ is irrational and $\beta \notin \alpha\mathbb{Z}$ then the characteristic function of $F_1(\alpha,\beta)$
is a minimal sequence \cite{Lawton1}. Their higher dimensional versions include quasicrystals such as the Penrose tilings
discussed in \cite{Bruijn}.
\end{rem}
\begin{lem}
\label{lem4}
If $\alpha \in \mathbb{R}, \frac{1}{4} \geq \beta > 0$ then $\Theta_{\alpha}$ maps
$\mathfrak{T}_1(F_1(\alpha,2\beta))$
\\
bijectively onto $\mathfrak{T}_2(F_2(\alpha,2\beta)),$
$\Theta_\alpha(\mathfrak{S}_{1}(F_1(\alpha,\beta))) = \mathfrak{S}_{2}(F_2(\alpha,\beta)),$
\\
and
$\Phi_\alpha(|p|^2) = |\Phi_\alpha(p)|^2, \ p \in \mathfrak{T}_{1}(F_1(\alpha,\beta)).$
\end{lem}
\begin{proof}
The first assertion follows since $j \rightarrow (j,\theta_\alpha(j))$ gives
a bijection of $F_1(\alpha,2\beta)$ onto $F_2(\alpha,2\beta)$ and
$\Theta_{\alpha}^{-1}(t)(x) = t(x,0), \, t \in \mathfrak{T}_2(F_2(\alpha,2\beta)),$
the second since
$\mathfrak{S}_{1}(F_1(\alpha,\beta)) \subseteq \mathfrak{\mathfrak{T}_1}(F_1(\alpha,2\beta)),$
and the third since $\theta_\alpha(m-n) = \theta_\alpha(m) - \theta_\alpha(n), \ m,n \in F_1(\alpha,\beta).$
\end{proof}
\begin{defin}
\label{propA}
$F \subset \mathbb{Z}^d$ has \emph{Property A} if $\mathfrak{U}_d(F) = \mathfrak{T}_{d}^{+}(F-F).$
\end{defin}
\begin{cor}
\label{cor1}
If $\alpha \in \mathbb{R}$ and $\frac{1}{4} \geq \beta > 0$ and $F_2(\alpha,\beta)$ has Property A and
$t \in \mathfrak{T}_1(F_1(\alpha,\beta)-F_1(\alpha,\beta))$ and
$\Theta_\alpha(t) \in \mathfrak{T}_{2}^{+}$ then
$t \in \mathfrak{U}_1(F_1(\alpha,\beta)).$
\end{cor}
\begin{proof}
Under the assumptions there exists a sequence $q_j \in \mathfrak{T}_2(F_2(\alpha,\beta))$
such that $|q_j|^2 \rightarrow \Theta_\alpha(t).$ Lemma \ref{lem1} implies that there
exists a sequence $p_j \in \mathfrak{T}_1(F_1(\alpha,\beta))$ such that $\Phi_\alpha(p_j) = q_j$
so $\Phi_\alpha(|p_j|^2) = |q_j|^2 \rightarrow \Theta_\alpha(t)$ hence $|p_j|^2 \rightarrow t$
so $t \in \mathfrak{U}_1(F_1(\alpha,\beta)).$
\end{proof}
\section{Generalized Spectral Factorization}
For $f \in C(\mathbb{T})$ with $|f| > 0$ we define the \emph{winding number}
\begin{equation}
\label{W}
    W(f) = \lim_{M \rightarrow \infty} \sum_{j=0}^{M-1} \log \left[\frac{f((j+1)/M)}{f(j/M)}\right].
\end{equation}
For $f \in C(\mathbb{T})$ with $|f| > 0$ and $W(f) = 0$ we define $Lf \in C(\mathbb{T})$ by
\begin{equation}
\label{L}
    (Lf)(x) = \log f(0) + \lim_{M \rightarrow \infty} \sum_{j=0}^{M-1} \log \left[\frac{f((j+1)x/M)}{f(jx/M)}\right], \ x \in \mathbb{T}.
\end{equation}
For $f \in A(\mathbb{T})$ we define $A_{\infty}^{\pm}f = \frac{1}{2}{\widehat f}(0) + \sum_{j=1}^{\infty} {\widehat f}(\pm j)\, e_{\pm j}.$
Clearly $W(fg) = W(f) + W(g),$ $L(fg) = L(f) + L(g),$ $\exp (Lf) = f,$ and if $\mathfrak{R}\, f > 0$ then $Lf = \log(f).$
If $f \in A(T)$ and $|f| > 0$ and $W(f) = 0$ then $Lf \in A(T).$
\begin{defin}
\label{gsfac}
For $f \in A(\mathbb{T})$ with $|f| > 0$ and $W(f) = 0$ let
\begin{equation}
\label{psi}
\Psi^{\pm}f = \exp ( A_{\infty}^{\pm} (Lf) ).
\end{equation}
\end{defin}
We record the following result whose proof is straightforward.
\begin{lem}
\label{gsfac2} If $f \in A(\mathbb{T})$ with $|f| > 0$ and $W(f) = 0$ then
$\Psi^{\pm}f \in A^{\pm}(\mathbb{T})$ and $(\Psi^{+}f)(\Psi^{-}f) = f.$
Furthermore, if $f > 0$ then $\Psi^{-}f = \overline {\Psi^{+}f}.$
\end{lem}
\begin{thm}
\label{thm1}
If $t \in \mathfrak{T}_1$ and $|t| > 0$ and $W(t) = 0$ then there exist \\
$\lambda_1,...,\lambda_{n^{-}(t)},\mu_1,...,\mu_{n^{+}(t)} \in \mathbb{D} \backslash \{0\}$
such that
\begin{equation}
\label{factor}
   t = e^{\gamma(t)} \, \prod_{j=1}^{n^{-}(t)} (1 - \lambda_j e_{-1}) \, \prod_{j=1}^{n^{+}(t)} (1 - \mu_j e_1).
\end{equation}
where
$\gamma(t) = \log t(0) - \sum_{j=1}^{n^{-}(t)} \log (1-\lambda_j) - \sum_{j=1}^{n^{+}(t)} \log (1-\mu_j).$
Then
\begin{equation}
\label{Psit}
\Psi^{+}t = e^{\gamma(t)/2} \, \prod_{j=1}^{n^{+}(t)} (1 - \mu_j e_1) \ \hbox{and} \ \Psi^{-}t =  e^{\gamma(t)/2} \, \prod_{j=1}^{n^{-}(t)} (1 - \lambda_j e_{-1}).
\end{equation}
\end{thm}
\begin{proof}
Since the polynomial
$P(z) = \sum_{j=-n^{-}(t)}^{n^{+}(t)} {\widehat t}(j) \, z^{n^{-}(t)+j}, \ \ z \in \mathbb{C},$
has degree $n^{-}(t) + n^{+}(t),$ $P(0) \neq 0,$ and $P(e_1) = e_{n^{-}(t)}\, t,$
it has $n^{-}(t) + n^{+}(t)$ roots, all roots $\neq 0$, and the modulus of every root $\neq 1.$
Hence there exist integers $m,n \geq 0$ with $m+n = n^{-}(t)+n^{+}(t)$ and
$\lambda_1,...,\lambda_{m},\mu_1,...,\mu_{n} \in \mathbb{D} \backslash \{0\}$
such that
$P(z) = \rho \, \prod_{j=1}^{m} (z - \lambda_j) \, \prod_{j=1}^{n} (1 - \mu_j \, z).$
Hence $W(P(e_1)) = m,$ and since
$P(e_1) = e_{n^{-}(t)},$
$W(P(e_1)) = W(e_{n^{-}(t)}) + W(t) = n^{-}(t).$
Therefore $m = n^{-}(t)$
so
$t = e_{-n^{-}(t)}P(e_1)$ which is (\ref{factor}) and (\ref{Psit}) then follows.
\end{proof}
\begin{cor}
\label{cor2} If $t \in \mathfrak{T}_1$ and $|t| > 0$ and $W(t) = 0$ and $N \geq n^{\pm}(t)$ then
\begin{equation}
\label{Psibound}
 ||\Psi^{\pm}t||_{\infty} \leq (1+\log (n^{\pm}(t)+1)) \,
 \exp \max (D_N* \log |t| + iH_N*\Im Lt)/2.
\end{equation}
\end{cor}
\begin{proof}
(\ref{e2}) and (\ref{psi}) $\Rightarrow$
$\Psi^{\pm}(t) = (\frac{1}{2}+A_{n^{\pm}(t)}^{\pm})*\exp (A_N*(Lt)).$
Since $D_N$ and $iH_N$ are real, $A_N = \frac{1}{2}(D_N \pm H_N),$ and $\Re Lt = \log |t|,$
$||\exp A_N*(Lt)||_{\infty} = \exp \max (D_N* \log |t| + iH_N*\Im Lt)/2.$
Therefore (\ref{Psibound}) follows from (\ref{e4}).
\end{proof}
\begin{rem} For odd $n \in \mathbb{Z}_+$ we let $\ell = (n-1)/2$ and construct the polynomial
$P_n(z) = (z-1/n)^{2n} -1$ and the trigonometric polynomial $t_n = e_{-n}P_n(e_1).$
Since the roots of $P_n$ inside $\mathbb{D}^{o}$ are
$-e^{\pi i k/n}+1/n, k = -\ell,...,\ell$ and the roots of $P_n$ outside
$\mathbb{D}$ are $e^{\pi i k/n}+1/n, k = -\ell,...,\ell$ and $P_n(1/n) = 1,$
a direct computation shows that $n^{\pm}(t) = n$ and
$$
\Psi^{-}t_n = \kappa \, \prod_{k=-\ell}^{\ell}  [1-(-e^{\pi i k/n}+1/n)e_{-1}], \
\Psi^{+}t_n = \kappa \, \prod_{k=-\ell}^{\ell}  [1-(e^{\pi i k/n}+1/n)^{-1}e_{1}]
$$
where $\kappa = \sqrt {1+1/n} \, \prod_{j=1}^{\ell} |e^{\pi i j/n} + 1/n| \approx \sqrt e.$
For large $n$, $\log |t_n| \in [-2n,3]$ yet $||\Psi^{\pm}t_n||_\infty \approx 2^n.$
The large norm of the spectral factor originates from the fact that $||\Im Lt_n||_{\infty} \approx \pi n/2.$ If $\Re \, t > 0$
then $\Im Lt = \Im \log t \in (-\pi/2,\pi/2)$ and we will obtain smaller bounds for $||\Psi^{\pm}t||_\infty.$
\end{rem}
\begin{defin}
For $t \in \mathfrak{T}_1$ that satisfies $\Re \, t > 0$ let
\begin{equation}
\label{rhosigma}
   \rho(t) =  \min \{ \  1, \ \min \, \Re \, t/(2\, e\, ||\widehat t||_1) \  \}, \ \ \sigma(t) = \rho(t)/n(t),
\end{equation}
\begin{equation}
\label{tau}
    \tau(t) = \max \ \{ \ \log \, (\, ||t||_{\infty} + \min \Re \, t/2), \ \pi/2 + |\log (\min \Re\, t/2)| \ \}.
\end{equation}
\end{defin}
\begin{lem}
\label{F1}
If $t \in \mathfrak{T}_1$ and $\Re \, t > 0$ then there exists $F \in H(\mathbb{A}_{\sigma(t)})$ such that $\log t = F(e_1).$
Furthermore,
$
   ||F||_{\infty} \leq \tau(t).
$
\end{lem}
\begin{proof}
Let $F = \log P$ where $P$ is the Laurent polynomial such that $t = P(e_1).$ If $z \in \mathbb{A}_{\sigma(t)}$
then $z = e^s \, e_1(x)$ where $|s| \leq \sigma(t)$ and $x \in \mathbb{T}.$ Therefore
$|P(z) - t(x)| \leq \min \Re \, t /2$
hence
$\Re \, P \geq \min \Re \, t/2$
so $F = \log P \in H(\mathbb{A}_{\sigma(t)})$ and $\log t = F(e_1)$ which proves the first assertion. The second assertion follows from the triangle inequality and (\ref{tau}).
\end{proof}
\begin{thm}
\label{thm2} There exists a function $\lambda : (0,\infty)^2 \rightarrow (0,\infty)$ such that
\begin{equation}
\label{psiasym}
    ||\Psi^{\pm}\, t||_\infty \leq \lambda(\rho(t),\tau(t))\, n(t)^{\pi/2}, \ \  t \in \mathfrak{T}_1, \ \Re \, t > 0.
\end{equation}
\end{thm}
\begin{proof} Let $\xi(N) = \max (D_N* \log |t| + iH_N*\Im Lt)/2$ and $q = \max \, |\Im t|.$
Since $\Re \, t > 0,$ $Lt = \log t$ so  $q < \pi/2$ and (\ref{e1}), (\ref{e5}) and Lemma (\ref{F1}) imply
$\xi(N) \leq n(t)\tau(t)e^{-N\rho(t)/n(t)}/\rho(t) + q \log N + (\log |t| + q)/2.$
Since the value of $N$ that minimizes the expression on the right is large when $n(t)$ is large,
it can be approximated by the root $N_0$ of the equation
$\tau(t)e^{-N_0\rho(t)/n(t)} = q/N_0.$
If we set $N_0 = n(t)(\log M)/\rho(t)$ then
$M/\log M = n(t)\tau(t)/q\rho(t)$ so \\
$M \approx n(t)\tau(t)/q\rho(t) \log [n(t)\tau(t)/q\rho(t)]$
and therefore \\
$\xi(N_0) \leq q(\log n(t) + \log \log M - \log \rho(t) + 1/2 + 1/\log M) + (\log |t|)/2.$
\\
This inequality and (\ref{thm2}) gives (\ref{psiasym}).
\end{proof}
\section{Two Dimensional Spectral Factorization}
Throughout this section $t \in \mathfrak{T}_2$ with $t > 0.$
Our objective is to approximate $t$ by functions
$|p|^2$ where $p \in \mathfrak{T}_2.$ For $\sigma > 0$ let $C_{\sigma}(\mathbb{T})$ denote the set of functions
in $C(\mathbb{T})$ that have the form $F(e_1)$ where $F \in H(\mathbb{A}_\sigma).$ Then
$C^{\omega}(\mathbb{T}) = \bigcup_{\sigma > 0} C_\sigma(\mathbb{T})$ and $C^{\omega}(\mathbb{T}) \bigotimes \mathfrak{T}_1$
is the subspace of $\mathfrak{T}_2$ of functions $f(x,y)$ that are trigonometric polynomials
in $y$ whose coefficients are real analytic functions of $x.$ For such $f$ we define
$f_N \in \mathfrak{T}_2$ by
\begin{equation}
\label{fN}
    f_N(x,y) = \int_{u \in \mathbb{T}} D_N(u)\, f(x-u,y) du, \ \ N \in \mathbb{Z}_+
\end{equation}
and observe that $||f-f_N||_{\infty} = \max_{(x,y) \in \mathbb{T}^2} |f(x,y)-f_N(x,y)| \rightarrow 0$
exponentially fast. We now construct a method to compute a spectral factorization $t = S^-(t)S^+(t)$ where $S^{\pm} \in C^{\omega}(T) \bigotimes \mathfrak{T}_1$ and $S^-(t) = \overline {S^+(t)}.$
\begin{defin}
\label{S1}
    For each $x \in \mathbb{T}$ define $h_x \in \mathfrak{T}_1$ by $h_x(y) = t(x,y), \, y \in \mathbb{T}.$
    Then define $S^{\pm}(t) \in C(\mathbb{T}^2)$ by
 \begin{equation}
 \label{S2}
    S^{\pm}(t)(x,y) = (\Psi^{\pm} h_x)(y).
 \end{equation}
\end{defin}
Since $t > 0$ it follows that $S^{\pm}(t) \in C^{\omega}(\mathbb{T}) \bigotimes \mathfrak{T}_1$ so $S_N^{\pm}(t) \rightarrow S^{\pm}(t)$
at an exponential rate. The remainder of this section computes that rate.
Every $t \in \mathfrak{T}_2$ admits a Fourier expansion
$t = \sum_{(j,k) \in \mathbb{Z}^2} {\widehat t}(j,k) \, e_{(j,k)},$
where $e_{(j,k)}$ are defined by
$e_{(j,k)}(x,y) = e_j(x)e_k(y).$ We define
$||\, {\widehat t}\, ||_1 = \sum_{(j,k) \in \mathbb{Z}^2} |{\widehat t}(j,k)|,$
$n_1(t) = \max \{|j| :{\widehat t}(j,k) \neq 0 \hbox{ for some } k \in \mathbb{Z} \, \},$
and
$n_2(t) = \max \{|k|:{\widehat t}(j,k) \neq 0 \hbox{ for some } j \in \mathbb{Z} \, \}.$
We define
$\rho(t) =  \min \{ \ 1, \ (\min t)/(2\, e\, ||\widehat t||_1) \ \},$
$\sigma_1(t) = \rho(t)/n_1(t),$
$\tau(t) = \max \ \{ \ \log (||t||_{\infty} + (\min t)/2), \ \pi/2 + |\log (\min t/2)| \ \}.$
The linear maps $\Gamma_z : \mathfrak{T}_2 \rightarrow \mathfrak{T}_1$
$
    \Gamma_z \, e_{j,k} = z^j \, e_k, \ \ z \in \mathbb{C}\backslash \{0\}
$
are algebra homomorphisms, $(\Gamma_{e_1(x)}\, t)(y) = t(x,y), \ \ x, y \in \mathbb{T},$ and for every
$s \in [0,\sigma_1(t)]$ and $x \in \mathbb{T},$
$||\Gamma_{e^se_1(x)} \, t - \Gamma_{e_1(x)} \, t||_{\infty}\leq \min t/2,$
and hence for $z \in \mathbb{A}_{\sigma_1(t)}$
\begin{equation}
\label{Gamma2}
\Re \, \Gamma_{z} \, t \geq \min t /2, \ \ \hbox{and} \ \  ||\Gamma_{z} \, t ||_{\infty} \leq ||t||_{\infty} + \min t/2.
\end{equation}
We observe that if
$z \in \mathbb{A}_{\sigma_1(t)}$
then
$\rho(\Gamma_z \, t) \geq \rho(t)/2e$ and $\tau(\Gamma_z \, t) \leq \tau(t) + \log 2$ and define
$\zeta(t) = \lambda(\rho(t)/2e,\tau(t) + \log 2).$
Then Theorem (\ref{psiasym}) and the fact that $n(\Gamma_z \, t) = n_2(t)$ give
\begin{equation}
\label{Gamma3}
||\Psi^{\pm} \, \Gamma_z \, t||_{\infty} < \zeta(t) \, n_2(t)^{\pi/2}.
\end{equation}
The fact that
$S^{\pm}(t)(x,y) = (\Psi^{\pm}\, \Gamma_{e_1(x)} \, t)(y)$
and the argument used to prove Lemma (\ref{lem1}), slighlty
generalized by considering $\Psi^{\pm} \, \Gamma_z \, t$ to be a function from $\mathbb{A}_{\sigma_1(t)}$ with values in
the normed subspace $\mathfrak{T}_1 \subset C(T),$ give
\begin{equation}
\label{Sbound1}
    ||S^{\pm}(t) - S_N^{\pm}(t)||_{\infty} \leq 2 \zeta(t) \, n_2(t)^{\pi/2} \, \sigma_1(t)^{-1} \, e^{-N\sigma_1(t)}.
\end{equation}
\begin{thm}
\label{thm3}
There exists a function $\varrho_1 : (0,\infty) \rightarrow (0,\infty)$ such that
\begin{equation}
\label{Sbound2}
    N > \varrho_1(\delta) n_1(t)(n_1(t)^\delta -\log \epsilon) \Rightarrow ||S^{\pm}(t) - S_N^{\pm}(t)||_{\infty} < \epsilon, \ \ \ \epsilon, \, \delta > 0.
\end{equation}
\end{thm}
\begin{proof} Define
$N_\epsilon = \frac{n_1(t)}{\rho(t)} \, [\log \, (2 \zeta(t) \, n_2(t)^{\pi/2} \, \sigma_1(t)^{-1}) - \log \epsilon)].$
Then (\ref{Sbound1}) implies that
$N \geq N_{\epsilon} \Rightarrow ||S^{\pm}(t) - S_N^{\pm}(t)||_{\infty} \leq \epsilon.$
The existence of the function $\varrho_1$ that satisfies (\ref{Sbound2}) then follows from the fact that for
$\epsilon, \delta > 0$ \\
$\lim_{N \rightarrow \infty} N_\epsilon / [n_1(t)(n_1(t)^{\delta} -\log \epsilon)] = 0.$
\end{proof}
\section{Derivation of Main Result}
This section derives the main result in this paper
\begin{thm}
\label{thm4}
If $\alpha$ is rational or if the Liouville-Roth constant $\mu_0(\alpha) > 2$
then for every $\beta > 0$ the set $F_2(\alpha,\beta)$ has property A.
\end{thm}
We recall that for $\alpha$ irrational $\mu_0(\alpha)$ is the least upper bound of the set of $\mu > 0$
for which there exists an infinite number of pairs $(c,d)$ of relatively prime integers such that
\begin{equation}
\label{L2}
    \left| \alpha + c/d \right| < \frac{1}{|d|^{\mu}}.
\end{equation}
$\mu_0(\alpha)$ is also called the \emph{irrationality measure}
of $\alpha$ and $\alpha$ is a Liouville number if $\mu_0(\alpha) = \infty.$
For $F \subseteq \mathbb{Z}^2$ define $F^{r} = \{\, (k,j) \, : \, (j,k) \in F \, \}.$
We observe that $F$ has property A if and only if $F^{r}$ has property A. Furthermore,
if $\alpha \neq 0$ then $F_2^{r}(\alpha,\beta) = F_2(1/\alpha,\beta/|\alpha|).$ So without loss of generality
we may assume that $|\alpha| \leq 1.$ We first prove the first assertion for the case that $\alpha$ is rational
by exploiting symmetries of $\mathbb{Z}^2.$ The modular group
$SL(2,\mathbb{Z}) = \{\, g = \left(
          \begin{array}{cc}
            a & b \\
            c & d \\
          \end{array}
        \right) \, : a,b,c,d \in \mathbb{Z}, ad-bc=1 \, \}$
acts as a group of group automorphisms on $\mathbb{Z}^2$ by
$g(j,k) = (aj+bk, cj+dk)$ and therefore as a group of involution preserving algebra automorphisms on $\mathfrak{T}_2$ by
$ge_{(j,k)} = e_{g(j,k)}.$ A direct computation gives
\begin{equation}
\label{L1}
    gF_2(\alpha,\beta) = F_2\left(\frac{c+d\alpha}{a+b\alpha},\frac{\beta}{|a+b\alpha|}\right), \ \ a+b\alpha \neq 0.
\end{equation}
If $\alpha$ is rational and $|\alpha| \leq 1$ then there exist coprime integers $c$ and $d$ such that $c+d\alpha = 0.$ Then
$|c| \leq |d|$ and there exist integers $a$ and $b$ such that $|a| \leq |b| \leq |d|/2$ and $ad-bc = 1.$ We let $g$ be the corresponding element in $SL(2,\mathbb{Z})$ and observe that
$gF_2(\alpha,\beta) = F_2(0,\beta |d|).$
Let
$t \in \mathfrak{T}_{2}^{+}(F_2(\alpha,\beta)-F_2(\alpha,\beta)).$
Then
$g(t) \in \mathfrak{T}_{2}^{+}(F_2(0,\beta|d|)-F_2(0,\beta|d|))$
so there exists
$k_1, k_2 \in \mathbb{Z}_{+}$ such that $k_1 < \beta|d|,$ $k_2 < \beta|d|,$ and
$\hbox{freq}(g(t)) \in \{-k_1-k_2,...,k_1+k_2\}.$ Then the spectral factor
$S^{+}(g(t)) \in C^{\omega}(\mathbb{T}) \bigotimes \mathfrak{T}_1(\{0,...,k_1+k_2\}).$
Therefore the sequence
$q_N = e_{(0,-k_2)}\, S_{N}^{+}(g(t)) \in \mathbb{T}_2(F_2(0,\beta|d|))$
and the sequence
$p_N = g^{-1}(q_N) \in \mathbb{T}_2(F_2(\alpha,\beta))$
satisfies
\begin{equation}
\label{converge1}
    || \, t - |p_N|^2 \, ||_{\infty} = ||\, g(t) - |q_n|^2 \, ||_{\infty} = || \, g(t) - |S_{N}^{+}(g(t))|^2 \, ||_{\infty}  \rightarrow 0
\end{equation}
since
$S_{N}^{+}(g(t)) \rightarrow S^{+}(g(t))$ and $g(t) = |S^{+}(t))^2.$ This proves the first assertion that $F_2(\alpha,\beta)$ has property A. We will now prove the second assertion by using Diophantine approximation of irrational $\alpha$ by rational numbers. First we need to compute an upper bound for the $|n_1(p_N)| + |n_2(p_N)|$ when $N$ is sufficiently large so that
approximation error $||S^{+}(t)-S_{N}^{+}(t)||_{\infty} < \epsilon.$ Theorem (\ref{thm3}) together with the properties
$g$ and $g^{-1}$ show that there exists a function $\varrho_2 : (0,\infty) \rightarrow (0,\infty)$ such that for every $\epsilon > 0$
there exists $N \in \mathbb{Z}_{+}$ such that $|| \, S^{+}(g(t)) - S_{N}^{+}(g(t)) \, ||_{\infty} < \epsilon$ and
\begin{equation}
\label{converge2}
    |n_1(p_N)| + |n_2(p_N)| \leq \varrho_2(\delta) |d|^2 (|d|^\delta -\log \epsilon), \ \ \ \delta > 0.
\end{equation}
Now assume that $\alpha$ be irrational, $\mu_{0}(\alpha) > 2,$ $\beta > 0,$ and
$t \in \mathfrak{T}_{2}^{+}(F_2(\alpha,\beta)-F_2(\alpha,\beta)).$
Then there exists $\beta > \widetilde \beta > 0$ such that
$t \in \mathfrak{T}_{2}^{+}(F_2(\alpha,\widetilde \beta)-F_2(\alpha,\widetilde \beta)).$
A direct computation shows that for every $g \in SL(2,\mathbb{Z})$
\begin{lem}
\label{lemfin}
    $(j,k) \in F_2(-c/d,\widetilde \beta)$ and $|j| < \frac{|d|(\beta - \tilde \beta)}{|c+d\alpha|}$ $\Rightarrow$
     $(j,k) \in F_2(\alpha,\beta).$
\end{lem}
\noindent If $\mu_0(\alpha) > 2$ there exists $\mu > 2$ and
a sequence of $g$ with
$$(\beta - \tilde \beta)/|\alpha + c/d| > (\beta-\tilde \beta) \, |d|^\mu.$$
This sequence grows faster than the expression $\varrho_2(\delta) |d|^2 (|d|^\delta -\log \epsilon)$
with $\delta = (\mu-2)/2$ and this proves the second assertion that $F_2(\alpha,\beta)$ has property A.
\begin{rem}
Does $F_2(\alpha,\beta)$ have property A if $\alpha$ is an irrational algebraic number?
Since the Thue-Siegal-Roth theorem shows that $\mu_0(\alpha) = 2$ our method of proof fails.
\end{rem}

\end{document}